\documentclass[12pt]{article}
\usepackage{amssymb}
\usepackage{amsmath}
\usepackage{amsthm}
\usepackage{xcolor}
\usepackage{booktabs}
\newtheorem{thm}{Theorem}
\newtheorem{prop}{Proposition}
\newtheorem{remark}{Remark}
\newtheorem{defi}{Definition}
\usepackage{amssymb,amsmath}
\usepackage{a4wide}
\numberwithin{equation}{section}
\usepackage{graphicx}
\usepackage{setspace}
\usepackage{natbib}
\usepackage{multicol}
\def\i{\mathrm{i}}
\usepackage{enumerate}

\title{A Frequency-Domain approach to detect nonstationarity in dependent data}
\author{ M. Ould Haye$^1$\footnote{corresponding author : {mohamedouhaye@cunet.carleton.ca}}\and
  A. Philippe$^2$}
\date{$^1$\small  School of Mathematics and Statistics. \\
Carleton University, 1125 Colonel By Dr. Ottawa, ON, Canada, K1S 5B6
\\ $^2$  Nantes Universit\'e, CNRS, Laboratoire de Math\'ematiques Jean Leray,\\ LMJL,
UMR 6629, F-44000 Nantes, France.}

\begin{document}
\maketitle
\begin{abstract}
Distinguishing long memory behaviour from nonstationarity can be very difficult as in both cases the sample autocovariance function decays very slowly. Available stationarity  tests either do not include long memory or fare poorly in terms of  empirical size, especially near the boundary between long memory and nonstationarity. 
We propose a testing procedure based on evaluating  periodograms at different epochs. Limiting distributions established here are easily tractable as sum of weighted independent $\chi^2$ random variables. Moreover, numerical studies are provided to show that the proposed approach seems to outperform existing methods.
\end{abstract}
\textbf{Keywords :} Long memory, nonstationarity, stationarity, Fourier frenquencies.

\section{Introduction}
An $L^2$-stationary process with long memory parameter $d$ generally defined as having a spectral density of the form
\begin{equation}\label{short}
f(\lambda)=\vert \lambda\vert^{-2d}f^*(\lambda)
\end{equation}
with $-1/2<d<1/2$ and $f^*$ is a positive, slowly varying function at zero. Long memory behaviour seems to be
often confused with a lack of stationarity, as time series exhibiting
such a behaviour tend to have a sample autocorrelation function with
large spikes at several lags which is well known to be the signature  of
non-stationarity for many practitioners.  There are many stationarity tests such as  \cite{pptest} tests and  the frequency test of \cite{MR3505787}, which can be seen as  extensions of the augmented Dickey-Fuller tests since their stationarity assumption  includes larger class than i.i.d. case. For instance the former allows for mixing and the latter includes  short memory linear processes. However, the first  statistical test for stationarity under a large umbrella
of dependence, i.e. short and long memory dependence versus non stationarity  is the so-called V/S test introduced by \cite{MR2328526}.\\
We develop a statistical test that uses an estimate  of the memory parameter $d$. Unlike the  V/S test, our test detects the stationarity for all $d$ in (-1/2,1/2) whereas the  V/S test requires $d$ to be in a compact $[a,b]\subset(-1/2,1/2)$. Another limitation in V/S consists in it  being oversized, with its empirical size significantly much larger than the nominal level, for $d\ge.4$, and therefore one cannot choose $b$ too close to 1/2. In contrast, the proposed test does not suffer form  this issue maintaining  empirical sizes close to the nominal level even for $d$  close to 1/2.  
From a sample $X_1,\ldots,X_n$ of the process $(X_t)$, we are
interested in building a testing procedure to discriminate between dependence and non-stationarity. \\
The proposed statistic is built from periodograms taken at different epochs. More precisely, the procedure is as follows:
we split our initial sample $X_1,\ldots,X_n$ into $m$ blocks (or
epochs), each of size $\ell$, and  we construct  the periodogram
$I_{n,i}$ on the $i$th block $X_{(i-1) \ell+1},\ldots,X_{i\ell }$.
Let
\begin{equation}\label{final}
Q_{n,m}(s,d)=\left(m^{-2d}\right)\sum_{j=1}^s \frac{I_n(\lambda_j)}{\frac{1}{m}\sum_{h=1}^mI_{n,h}(\lambda'_j)},
\end{equation}
where $\lambda_j=2\pi j/n$, and  $\lambda'_j=2\pi j/\ell$, $j=1,\ldots,s$, for fixed $s$, are the Fourier frequencies, and 
\begin{equation}\label{perio}
I_n(\lambda_j)=\frac{1}{2\pi n}\left\vert\sum_{t=1}^nX_te^{\i t\lambda_j}\right\vert^2,\qquad I_{n,h}(\lambda_j')=\frac{1}{2\pi \ell}\left\vert\sum_{t=(h-1)\ell+1}^{h\ell}X_te^{\i t\lambda_j'}\right\vert^2.
\end{equation}
The statistic $Q_{n,m}(s,d)$ can be viewed as sum of self normalized periodograms $I_n(\lambda_j)$, $j=1,\ldots,s$.
In  our 2018 paper (\cite{Gromykov}), we considered the statistic $Q_{n,m}(s,0)$ to test short memory versus long memory, and we only established its asymptotic distribution  under short memory $d=0$. It is worth noticing that establishing limiting distributions under long or anticipative memory ($-1/2<d<1/2$) requires  modifying the short memory statistic, including a normalizing coefficient $m^{-2d}$, as it appears in (\ref{final}), as well as two sets of Fourier frequencies $\lambda_j$ and $\lambda'_j$. Of course, this will result in different limiting distributions due to the long/anticipative memory that will strongly affect the asymptotic behaviour of the properly normalized periodograms in (\ref{perio}).    

We assume that $m=m(n),\textrm{ and } \ell=\ell(n)\to\infty$ as $n\to\infty$. It is  important to emphasize the fact that $m$ and $\ell$ increase with $n$ and are not constant, and that $m=n/\ell$.
We are simply using notation $m$ and $\ell$ rather than $m(n)$ and
$\ell(n)$ only for the sake of simplicity. 
The rest of the paper is organized as follows. 
Section 2 contains main limiting theorems related to the statistic $Q_{n,m}(s,d)$, including its behaviour under a wide range of nonstationary alternatives. 
Moreover, we propose a  test statistic to detect nonstationarity  based on the asymptotic properties of the statistic $Q_{n,m}(s,d)$.
Section 3 contains some  Monte
Carlo study to illustrate the performance of our proposed method compared to V/S.
Section 4  contains  all the proofs.

\section{Main results}\label{sec:2}
 Let $(X_t)_{t\ge0}$ be a linear process of the form
\begin{equation}\label{linear}X_t =
  \sum_{j=0}^\infty a_j \epsilon_{t-j},
\end{equation}
where $a_j$ are square summable,  $a_0=1$, and  
 where $(\epsilon_j)$ are i.i.d. random variables with zero mean and finite fourth moment.
 Denote $\sigma^2=\mathbb{E}(\epsilon_1^2)$.
 \begin{defi}~\\
 $\bullet$
We say that the process $X_t$ is an $I(d)$ with $-1/2\le d<1/2$ if it  is a linear  process as defined in (\ref{linear}) with coefficients $(a_j)_j$ satisfying
 \begin{align}
 & \sum_{j=0}^\infty\vert a_j\vert<\infty && \text{ if } d=0
  \label{summable} \\
  &  a_j\sim c(d)\,j^{-1+d}   && \text{ if }  0< d < 1/2  \label{positived}\\
 \sum_{j=0}^\infty a_j=0 & \text{ and }  a_j=c(d)\,j^{-1+d}(1+O(j^{-1}))  && \text{ if } -1/2 < d < 0   \label{negatived1}
\end{align}
where $c(d)>0$.\\
$\bullet$ We  say that a stochastic process $X_t$ is an $I(d)$ with $1/2\le d<3/2$ if $X_t-X_{t-1}$ is an $I(d-1)$ process.
\end{defi}
Condition (\ref{summable}), imposed when $d=0$, implies that the covariance function $\gamma(h)$ is summable, i.e.,
 $$
 \sum_{h=-\infty}^\infty\vert\gamma(h)\vert<\infty.
 $$
 According to Proposition 3.2.1. of \cite{MR2977317}, when $d\neq 0$ the conditions \eqref{positived} and   (\ref{negatived1}) implies that the \begin{equation}\label{gammad}
\gamma(h)\sim C(d) h^{2d-1},\qquad\textrm{as }h\to\infty,\qquad\textrm{for some constant }C(d) >0,
 \end{equation}
 where $a_n\sim b_n$ means that $a_n/b_n\to1$ as $n\to\infty$.
 
For $I(d)$ processes ($-1/2<d<3/2$), \cite{OuldHayePhilippe2026}   prove the convergence of the discrete Fourier transforms (DFT) at first   Fourier frequencies  $\lambda_1,\ldots,\lambda_s$, with fixed $s\ge1$: 
\begin{eqnarray}\label{dft}\lefteqn{
Z_n(s,d):=\Bigg[\left(\frac{1}{n^{1/2+d}}\sum_{k=1}^n\cos(k\lambda_{1})X_k\right),
\cdots,\left(\frac{1}{n^{1/2+d}}\sum_{k=1}^n\cos(k\lambda_{s})X_k\right),}\nonumber\\
&& \left(\frac{1}{n^{1/2+d}}\sum_{k=1}^n\sin(k\lambda_1)X_k\right),\cdots,\left(\frac{1}{n^{1/2+d}}\sum_{k=1}^n\sin(k\lambda_{s})X_k\right)\Bigg]\overset{d}{\rightarrow}\mathcal{N}\left(0,\Sigma(d)\right), 
\end{eqnarray}
where the covariance matrix  $\Sigma(d)$ is defined by 
\begin{equation}\label{Sigma}
\Sigma(d) =
\begin{pmatrix}
\Sigma^{(c)}(d) & 0\\[3pt]
0 &\Sigma^{(s)}(d)
\end{pmatrix},
\end{equation}
where for 
 $i,j=1,\ldots,s$, 
\[
\Sigma^{(c)}_{ij}(d) =
\begin{cases}
\displaystyle \frac{\delta(-1/2)}{2}\int_{[0,1]^2}\!\cos(2\pi i x)\cos(2\pi j y)\,
\bigl(-\log|x-y|\bigr)\,dx\,dy,
& \text{if } d=\tfrac{1}{2},\\[10pt]
\displaystyle\frac{\delta(d-1)}{2}\int_{[0,1]^2}\!\cos(2\pi i x)\cos(2\pi j y)\,
|x-y|^{2d-1}\,dx\,dy,
& \text{if } 1/2<d<3/2,\\[10pt]
-\delta(d)\Bigg[a_{ij}(d)+
2\pi^2 i j\displaystyle\int_{[0,1]^2}\sin(2\pi ix)\sin(2\pi jy)\vert x- y\vert^{2d+1}dxdy\Bigg], 
& \text{if } -1/2<d<1/2,
\end{cases}
\]
where
$$
a_{ij}(d)=1-(2d+1)\int_0^1x^{2d}\left(\cos(2\pi ix)+\cos(2\pi jx)\right)dx,
$$
and
\begin{equation}\label{const delta}
\delta(d)=\begin{cases}8\sigma^2_\epsilon c^2(-1/2),&\textrm{if }d=-1/2,\\\\
\sigma^2_\epsilon\left(\displaystyle\sum_{i=0}^\infty a_i\right)^2,&\textrm{if }d=0,\\\\
   \frac{C(d)}{d(2d+1)},&\textrm{if }0<\vert d\vert<1/2.
\end{cases}
\end{equation}
Similarly, 
\[
\Sigma^{(s)}_{ij}(d) =
\begin{cases}
\displaystyle\frac{\delta(-1/2)}{2} \int_{[0,1]^2}\!\sin(2\pi i x)\sin(2\pi j y)\,
\bigl(-\log|x-y|\bigr)\,dx\,dy,
& \text{if } d=\tfrac{1}{2},\\[10pt]
\frac{\delta(d-1)}{2}\displaystyle\int_{[0,1]^2}\!\sin(2\pi i x)\sin(2\pi j y)\,
|x-y|^{2d-1}\,dx\,dy,
& \text{if } 1/2<d<3/2,\\[10pt]
-\delta(d)(2\pi^2 ij)\displaystyle\int_{[0,1]^2}\!\cos(2\pi i x)\cos(2\pi j y)\,
|x-y|^{2d+1}\,dx\,dy,
& \text{if } -1/2<d<1/2.
\end{cases}
\]

In what  follows, we will focus on the statistic $Q_{n,m}(s,d)$ defined in (\ref{final}). In the next theorem  we give the  asymptotic
distribution of $Q_{n,m}(s,d)$ for $I(d)$ processes.
 \begin{thm}\label{Th1}
Let $(X_t)$ be an $I(d)$ process with memory parameter $d \in (-1/2,\,3/2)$. Then the statistic $Q_{n,m}(s,d)$ defined in (\ref{final}) satisfies
\begin{equation}\label{d limit}
Q_{n,m}(s,d) \;\overset{\mathcal{D}}{\longrightarrow}\;
Q(s,d) = \sum_{i=1}^{2s} \zeta_i(d)\, Q_i, 
\end{equation}
where $\{Q_i\}$ are i.i.d.\ $\chi^2(1)$ random variables, and $\{\zeta_i(d)\}$ are the eigenvalues of the  matrix
$\Sigma(d)D^{-1}(d)$, where
$D(d)$ is the $2s\times 2s$  diagonal matrix with diagonal entries
\begin{equation} \label{D}
D_{ii}(d)=D_{i+s,i+s}(d)=\left(\Sigma^{(c)}_{ii}(d) + \Sigma^{(s)}_{ii}(d)\right),\qquad i=1,\ldots,s.\end{equation}
\end{thm}
The following proposition, shows continuity of the limiting distribution as a function of $d$ at the border $d=1/2$:
\begin{prop}\label{continuity1}
   The function 
   $$
   d\mapsto\Sigma(d)D^{-1}(d)
   $$ 
   is continuous at $d=1/2.$ 
\end{prop}
 \begin{prop}\label{prop2}
Let   $Y_t$ be  an $I(d)$ process with  $d\in(-1/2,1/2)$, 
and consider   a nonstationary process with deterministic  trend $X_t$:   
\begin{equation}\label{trend d}X_t=g_n(t)+Y_t\end{equation} where $g_n(t)=n^\beta g(t/n)$, with $\beta\ge0$, and where $g$ is a piece-wise continuous\footnote{the interval [0,1] can be divided into a finite number of sub-intervals such that the function is continuous on each open sub-interval, with finite limit at each endpoint.}  satisfying  
\begin{equation}\label{nonzero}
\int_0^1g(x)e^{\i 2\pi jx}dx\neq0, \qquad\textrm{for some }j\in\{1,2,\ldots,s\}.
\end{equation}
Then, for all $\delta \leq 1/2 $  
\begin{equation} \label{AsympH1} Q_{n,m}(s,\delta)\overset{P}{\longrightarrow}\infty, \end{equation}
where $\overset{P}{\longrightarrow}$ denotes the convergence in probability. 
 \end{prop}
 \begin{proof}
 The proof is relegated to Section  \ref{ap:proof2} of the appendix.
\end{proof}

 \begin{remark}\label{rem2}
We note in passing that condition (\ref{nonzero}) is satisfied by most of nonconstant functions. For example, nonnegative step functions (allowing change in the mean), polynomials, etc.
\end{remark}

We want to test the stationarity of  $I(d)$ processes (that is $-1/2<d<1/2$) against  the class of nonstationary processes $I(d)$, $1/2<d<3/2$ and deterministic trend \eqref{trend d}. That is,  let $H_0$ be the stationarity assumption and $H_1$ the nonstationarity one.
This is the same context as VS test (see \cite{MR2328526}.  The fact that our limiting distribution is well defined for  $d=1/2$ gives our test a considerable advantage over V/S. Indeed, as the limiting distribution of V/S is degenerated at $d=1/2$,  
   $\widehat d$ had to be restricted to a compact $[-a,a]\subset(-1/2,1/2)$ in V/S and its implementation. This leads to a blind spot $(a,1/2)$ in the stationarity null hypothesis. This explains V/S's high empirical sizes observed when  $d$ is close to 1/2 (see simulation results in \cite{MR2328526} and the following Section \ref{simulations}). 

\begin{prop}\label{prop22}
Let $q_\alpha(s,d)$ be $Q(s,d)$'s quantile of order $1-\alpha$ where $0<\alpha<1$. For fixed $\alpha$, the mapping $(-1/2,1/2]\to\mathbb{R}$, $d\mapsto q_\alpha(s,d)$ is continuous.
\end{prop}

We now describe the construction of a critical region $R_n$ derived from the statistic $Q_{n,m}(s,d)$ to obtain a testing procedure with asymptotic significance level $\alpha$:
$$
\alpha=\underset{-1/2<d<1/2}{\sup}\,\,\underset{n\to\infty}{\lim}P_d\left((X_1,\ldots,X_n)\in R_n\right),
$$
where $P_d$ is the probability distribution when the observations come from  an   $I(d)$ process.
 Deriving a decision rule from the statistic $Q_{n,m}(s,d)$ requires the estimation of the parameter $d$. 
We consider local Whittle estimator of the memory parameter $d$ defined by
\begin{equation}\label{lwe}
\widehat{d}_n=\underset{\delta\in (-1/2 ,1/2]}{\operatorname{argmin}}\,\,U_n(\delta) 
\end{equation}
where the contrast function $U_n$ is defined by
\[U_n(\delta)=\log\left(\frac{1}{b_n}\sum_{j=1}^{b_n}\lambda_j^{2\delta}I_n(\lambda_j)\right)-\frac{2\delta}{b_n}\sum_{j=1}^{b_n}\log\lambda_j,\]
 and the bandwidth parameter  satisfies $b_n\to \infty$ and $b_n=o(n)$.

For $d\in (-1/2 , 1/2)$, if $\mathbb{E}(\epsilon_0^4)<\infty$ and   if there exists $\beta\in(0,2]$,  $c_0>0$ and $c_1\neq0$ such that
 $$
 f^*(\lambda)=c_0+c_1\vert\lambda\vert^\beta+o\left(\vert\lambda\vert^\beta\right),\qquad\textrm{as }\lambda\to0,
 $$
 then $\widehat d$ is  a $\log(n)$-consistent estimator of $d$; see Corollary 1 and Proposition 5 of \cite{Dalla}.\\
 \begin{prop}\label{empirical size 1}
Consider the following critical region 
  \begin{equation}\label{RNN}
  R_n=\{(X_1,\ldots,X_n)\in\mathbb{R}^n,\quad Q_{m,n}(s,\widehat d)>q_\alpha(s,\widehat d)\},
  \end{equation}
  where $q_\alpha(s,\widehat d)$ is the quantile of order ($1-\alpha)$ of $Q(s,\widehat d)$    The test based on $R_n$ has an asymptotic significance level $\alpha$ and is consistent under the alternative $H_1.$
  \end{prop}

  \section{Finite Sample  Performance}\label{simulations}
\subsection{Choice of the tuning parameters $m,s$ }
The simulation study investigates the influence of the tuning parameters
$s$ and $m=\lfloor n^\gamma\rfloor$ on both the empirical size and the power
of the proposed test. Tables \ref{tab:H0-500} and \ref{tab:H0-5000} investigate empirical sizes under various choices of $m=[n^\gamma]$ and $s$ for two sample sizes; $n=500$ and $n=5000$ respectively. Empirical powers are also compared  for different values of these tuning parameters in Tables \ref{tab:H1-500} and \ref{tab:H1-5000}.
Overall, these simulations indicate that the block exponent $\gamma$ has a much
larger impact on the empirical size than the parameter $s$. The presence of a
positive autoregressive component mainly affects the choice of the block size,
whereas the influence of $s$ remains comparatively moderate.

Taking these observations together, we recommend using
$m=\lfloor n^{1/2}\rfloor$ as a default choice. Regarding the parameter $s$,
the simulations suggest selecting $s=1$ for moderate sample sizes
($n=500$), while increasing to $s=2$ for larger samples ($n=5000$).
For $n=500$, this choice preserves a satisfactory empirical size, whereas
for $n=5000$ it yields a substantial gain in power with only a limited impact
on the empirical level, even in the presence of a short-range autoregressive
component.

Therefore, our recommended default specification is
\begin{equation}\label{tuning choice}
(s,m)=
\begin{cases}
(1,\lfloor n^{1/2}\rfloor), & n=500,\\[1mm]
(2,\lfloor n^{1/2}\rfloor), & n=5000.
\end{cases}
\end{equation}

This recommendation is not intended to be optimal for every possible
dependence structure, but rather to provide a practical compromise between
empirical size and power over the range of models considered in this study.

\begin{table}[ht]
\centering
\small
\begin{tabular}{c|ccccc|ccccc|ccccc}
\toprule
& \multicolumn{5}{c|}{$s=1$}
& \multicolumn{5}{c|}{$s=2$}
& \multicolumn{5}{c}{$s=3$}\\
\cmidrule(lr){2-6}
\cmidrule(lr){7-11}
\cmidrule(lr){12-16}
$d$
& \multicolumn{5}{c|}{$\gamma$}
& \multicolumn{5}{c|}{$\gamma$}
& \multicolumn{5}{c}{$\gamma$}\\
& $.3$ & $.4$ & $.5$ & $.6$ & $.7$
& $.3$ & $.4$ & $.5$ & $.6$ & $.7$
& $.3$ & $.4$ & $.5$ & $.6$ & $.7$\\
\midrule
\multicolumn{16}{c}{\textbf{ARFIMA(0,$d$,0) }}\\
\midrule
$-.25$
& \textcolor{red}{.082} & \textcolor{red}{.065} & .038 & .019 & .043
& \textcolor{red}{.091} & .050 & .016 & .025 & .056
& \textcolor{red}{.096} & .039 & .016 & .027 & \textcolor{red}{.085}\\

$0$
& \textcolor{red}{.094} & .053 & .027 & .015 & .032
& \textcolor{red}{.097} & .049 & .013 & .014 & .037
& \textcolor{red}{.104} & .035 & .018 & .029 & .055\\

$.25$
& \textcolor{red}{.070} & .045 & .036 & .016 & .015
& \textcolor{red}{.067} & .048 & .018 & .013 & .015
& \textcolor{red}{.072} & .029 & .014 & .021 & .024\\

$.35$
& \textcolor{red}{.077} & .059 & .042 & .025 & .020
& \textcolor{red}{.089} & .058 & .023 & .014 & .022
& \textcolor{red}{.084} & .038 & .015 & .011 & .012\\

$.45$
& \textcolor{red}{.079} & \textcolor{red}{.061} & .040 & .023 & .019
& \textcolor{red}{.073} & \textcolor{red}{.060} & .027 & .021 & .012
& \textcolor{red}{.080} & .043 & .023 & .012 & .007\\

$.49$
& \textcolor{red}{.089} & \textcolor{red}{.085} & \textcolor{red}{.067}
& .053 & .046
& \textcolor{red}{.099} & \textcolor{red}{.073}
& \textcolor{red}{.063} & .044 & .029
& \textcolor{red}{.108} & \textcolor{red}{.081}
& .058 & .044 & .018\\
\midrule
\multicolumn{16}{c}{\textbf{ARFIMA(1,$d$,0) with AR coefficient $\phi=0.7$}}\\
\midrule
$-.25$
& .011 & .002 & .000 & .001 & .047
& .013 & .000 & .000 & .045 & \textcolor{red}{.241}
& .010 & .002 & .020 & \textcolor{red}{.181} & \textcolor{red}{.454}\\

$0$
& .016 & .003 & .000 & .000 & .022
& .011 & .000 & .000 & .042 & \textcolor{red}{.177}
& .012 & .001 & .018 & \textcolor{red}{.163} & \textcolor{red}{.340}\\

$.25$
& .009 & .003 & .002 & .000 & .018
& .010 & .003 & .005 & .031 & \textcolor{red}{.116}
& .012 & .004 & .023 & \textcolor{red}{.114} & \textcolor{red}{.200}\\

$.35$
& .029 & .016 & .011 & .035 & \textcolor{red}{.111}
& .027 & .020 & .052 & \textcolor{red}{.143} & \textcolor{red}{.326}
& .040 & .041 & \textcolor{red}{.157} & \textcolor{red}{.357} & \textcolor{red}{.493}\\

$.45$
& .052 & \textcolor{red}{.060} & \textcolor{red}{.076}
& \textcolor{red}{.143} & \textcolor{red}{.339}
& \textcolor{red}{.078} & \textcolor{red}{.101}
& \textcolor{red}{.217} & \textcolor{red}{.420}
& \textcolor{red}{.659}
& \textcolor{red}{.100} & \textcolor{red}{.175}
& \textcolor{red}{.407} & \textcolor{red}{.696}
& \textcolor{red}{.822}\\

$.49$
& \textcolor{red}{.081} & \textcolor{red}{.092}
& \textcolor{red}{.122} & \textcolor{red}{.245}
& \textcolor{red}{.453}
& \textcolor{red}{.110} & \textcolor{red}{.149}
& \textcolor{red}{.304} & \textcolor{red}{.557}
& \textcolor{red}{.771}
& \textcolor{red}{.158} & \textcolor{red}{.272}
& \textcolor{red}{.549} & \textcolor{red}{.770}
& \textcolor{red}{.878}\\
\bottomrule
\multicolumn{16}{c}{\textbf{ARFIMA(1,$d$,0) with AR coefficient $\phi=-0.7$}}\\
\midrule
$-.25$
& \textcolor{red}{.073} & .039 & .014 & .005 & .003
& \textcolor{red}{.081} & .035 & .006 & .000 & .001
& \textcolor{red}{.089} & .025 & .005 & .000 & .000\\

$0$
& \textcolor{red}{.082} & \textcolor{red}{.060} & .032 & .011 & .001
& \textcolor{red}{.087} & .051 & .016 & .006 & .000
& \textcolor{red}{.094} & .039 & .014 & .002 & .000\\

$.25$
& \textcolor{red}{.082} & \textcolor{red}{.072} & .039 & .018 & .004
& \textcolor{red}{.096} & .058 & .013 & .006 & .000
& \textcolor{red}{.103} & .048 & .011 & .001 & .000\\

$.35$
& \textcolor{red}{.082} & \textcolor{red}{.066} & .036 & .014 & .004
& \textcolor{red}{.090} & \textcolor{red}{.065} & .013 & .004 & .000
& \textcolor{red}{.106} & .053 & .016 & .000 & .000\\

$.45$
& \textcolor{red}{.093} & .057 & .047 & .021 & .001
& \textcolor{red}{.101} & \textcolor{red}{.061} & .023 & .009 & .000
& \textcolor{red}{.099} & .050 & .019 & .000 & .000\\

$.49$
& \textcolor{red}{.090} & \textcolor{red}{.081} & \textcolor{red}{.060}
& .034 & .010
& \textcolor{red}{.117} & \textcolor{red}{.076}
& .041 & .023 & .001
& \textcolor{red}{.120} & \textcolor{red}{.081}
& .036 & .010 & .000\\
\bottomrule
\end{tabular}
\caption{Empirical rejection probabilities, at nominal level $\alpha=5\%$, for ARFIMA$(0,$d$,0)$ and ARFIMA(1,d,0) for  $n=500$, $s\in\{1,2,3\}$ and $m=\lfloor n^\gamma\rfloor$, where $\gamma\in\{.3,.4,.5,.6,.7\}$. Number of replications is 1000. Values greater than or equal to $6\%$ are shown in red.}
\label{tab:H0-500}
\end{table}
\begin{table}[ht]
\centering
\small
\begin{tabular}{c|ccccc|ccccc|ccccc}
\toprule
& \multicolumn{5}{c|}{$s=1$}
& \multicolumn{5}{c|}{$s=2$}
& \multicolumn{5}{c}{$s=3$}\\
\cmidrule(lr){2-6}
\cmidrule(lr){7-11}
\cmidrule(lr){12-16}
$d$
& \multicolumn{5}{c|}{$\gamma$}
& \multicolumn{5}{c|}{$\gamma$}
& \multicolumn{5}{c}{$\gamma$}\\
& $.3$ & $.4$ & $.5$ & $.6$ & $.7$
& $.3$ & $.4$ & $.5$ & $.6$ & $.7$
& $.3$ & $.4$ & $.5$ & $.6$ & $.7$\\
\midrule
\multicolumn{16}{c}{\textbf{ARFIMA(0,$d$,0) }}\\
\midrule
$-.25$
& \textcolor{red}{.070} & .050 & .055 & .045 & .058
& \textcolor{red}{.073} & .048 & .039 & .033 & .050
& \textcolor{red}{.078} & .038 & .029 & .029 & .036\\

$0$
& \textcolor{red}{.077} & \textcolor{red}{.065} & .052 & .044 & .052
& \textcolor{red}{.079} & \textcolor{red}{.061} & .040 & .042 & .044
& \textcolor{red}{.080} & .052 & .029 & .034 & .044\\

$.25$
& \textcolor{red}{.081} & .057 & .054 & .040 & .041
& \textcolor{red}{.093} & .059 & .045 & .035 & .029
& \textcolor{red}{.097} & .054 & .039 & .041 & .038\\

$.35$
& .054 & .048 & .045 & .036 & .038
& \textcolor{red}{.069} & .051 & .034 & .029 & .030
& \textcolor{red}{.069} & .046 & .025 & .027 & .020\\

$.45$
& \textcolor{red}{.064} & .054 & .044 & .036 & .035
& \textcolor{red}{.077} & .049 & .038 & .024 & .022
& \textcolor{red}{.069} & .033 & .025 & .020 & .016\\

$.49$
& \textcolor{red}{.060} & .054 & .049 & .043 & .036
& \textcolor{red}{.070} & .054 & .044 & .039 & .033
& \textcolor{red}{.073} & .043 & .040 & .032 & .024\\
\midrule
\multicolumn{16}{c}{\textbf{FARIMA(1,0,0) with  AR coefficient $\phi=0.7$}}\\
\midrule
$-.25$
& .021 & .005 & .003 & .004 & \textcolor{red}{.113}
& .024 & .003 & .003 & .026 & \textcolor{red}{.411}
& .014 & .004 & .005 & \textcolor{red}{.113} & \textcolor{red}{.707}\\

$0$
& .030 & .015 & .011 & .008 & \textcolor{red}{.077}
& .017 & .005 & .002 & .026 & \textcolor{red}{.334}
& .016 & .006 & .004 & \textcolor{red}{.095} & \textcolor{red}{.663}\\

$.25$
& .013 & .007 & .002 & .003 & .030
& .014 & .010 & .002 & .022 & \textcolor{red}{.237}
& .011 & .001 & .000 & \textcolor{red}{.072} & \textcolor{red}{.541}\\

$.35$
& .024 & .009 & .001 & .001 & .029
& .016 & .004 & .004 & .011 & \textcolor{red}{.171}
& .016 & .004 & .004 & \textcolor{red}{.060} & \textcolor{red}{.448}\\

$.45$
& .035 & .021 & .014 & .016 & \textcolor{red}{.090}
& .034 & .017 & .020 & \textcolor{red}{.063} & \textcolor{red}{.311}
& .029 & .023 & .026 & \textcolor{red}{.136} & \textcolor{red}{.580}\\

$.49$
& \textcolor{red}{.066} & \textcolor{red}{.063} & .059
& \textcolor{red}{.085} & \textcolor{red}{.199}
& \textcolor{red}{.076} & \textcolor{red}{.061}
& \textcolor{red}{.073} & \textcolor{red}{.172}
& \textcolor{red}{.516}
& \textcolor{red}{.073} & \textcolor{red}{.064}
& \textcolor{red}{.097} & \textcolor{red}{.346}
& \textcolor{red}{.787}\\
\midrule
\multicolumn{16}{c}{\textbf{ARFIMA$(1,d,0)$ with AR coefficient $\phi=-0.7$}}\\
\midrule
$-.25$
& \textcolor{red}{.062} & .048 & .029 & .017 & .005
& \textcolor{red}{.078} & .044 & .029 & .020 & .007
& \textcolor{red}{.092} & .041 & .020 & .014 & .000\\

$0$
& \textcolor{red}{.068} & .050 & .040 & .030 & .017
& \textcolor{red}{.070} & .044 & .036 & .018 & .003
& \textcolor{red}{.075} & .046 & .019 & .012 & .002\\

$.25$
& \textcolor{red}{.090} & \textcolor{red}{.065} & .054 & .046 & .021
& \textcolor{red}{.077} & .059 & .036 & .024 & .006
& \textcolor{red}{.087} & .056 & .033 & .021 & .001\\

$.35$
& \textcolor{red}{.069} & .057 & .046 & .042 & .026
& \textcolor{red}{.086} & \textcolor{red}{.065} & .043 & .031 & .011
& \textcolor{red}{.072} & .051 & .032 & .018 & .002\\

$.45$
& \textcolor{red}{.062} & .047 & .033 & .025 & .015
& .055 & .036 & .027 & .016 & .006
& .051 & .030 & .017 & .008 & .001\\

$.49$
& \textcolor{red}{.068} & \textcolor{red}{.064} & .056 & .052 & .036
& \textcolor{red}{.076} & \textcolor{red}{.071} & \textcolor{red}{.062}
& .049 & .017
& \textcolor{red}{.092} & \textcolor{red}{.072}
& \textcolor{red}{.064} & .049 & .006\\
\bottomrule
\end{tabular}
\caption{Empirical rejection probabilities, at nominal level $\alpha=5\%$, for ARFIMA$(0,d,0)$ and ARFIMA(1,d,0) for  $n=5000$, $s\in\{1,2,3\}$ and $m=\lfloor n^\gamma\rfloor$, where $\gamma\in\{.3,.4,.5,.6,.7\}$. Number of replications is 1000. Values greater than or equal to $6\%$ are shown in red.}
\label{tab:H0-5000}
\end{table}

\begin{table}[ht]
\centering
\small
\begin{tabular}{c|ccccc|ccccc|ccccc}
\toprule
& \multicolumn{5}{c|}{$s=1$}
& \multicolumn{5}{c|}{$s=2$}
& \multicolumn{5}{c}{$s=3$}\\
\cmidrule(lr){2-6}
\cmidrule(lr){7-11}
\cmidrule(lr){12-16}
$d$
& \multicolumn{5}{c|}{$\gamma$}
& \multicolumn{5}{c|}{$\gamma$}
& \multicolumn{5}{c}{$\gamma$}\\
& $.3$ & $.4$ & $.5$ & $.6$ & $.7$
& $.3$ & $.4$ & $.5$ & $.6$ & $.7$
& $.3$ & $.4$ & $.5$ & $.6$ & $.7$\\
\midrule
$.51$
& .101 & .085 & .083 & .077 & .068
& .102 & .110 & .084 & .080 & .053
& .114 & .108 & .082 & .067 & .033\\

$.55$
& .108 & .112 & .109 & .118 & .114
& .139 & .133 & .137 & .132 & .121
& .152 & .162 & .155 & .163 & .095\\

$.65$
& .185 & .249 & .301 & .363 & .383
& .283 & .354 & .435 & .521 & .548
& .346 & .437 & .516 & .585 & .564\\

$.75$
& .329 & .451 & .532 & .631 & .704
& .476 & .592 & .707 & .810 & .849
& .542 & .704 & .805 & .887 & .894\\

$1$
& .613 & .765 & .858 & .918 & .950
& .785 & .931 & .971 & .990 & .996
& .905 & .978 & .991 & .998 & .999\\

$1.25$
& .821 & .922 & .973 & .988 & .997
& .951 & .990 & .997 & .999 & 1
& .986 & .997 & .998 & .999 & 1\\
\bottomrule
\end{tabular}
\caption{Empirical rejection probabilities, at nominal level $\alpha=5\%$, for FARIMA(0,$d$,0) under $H_1$ with $n=500$, 
$s\in\{1,2,3\}$ and
$m=\lfloor n^\gamma\rfloor$, where
$\gamma\in\{.3,.4,.5,.6,.7\}$. Number of replications is 1000.}
\label{tab:H1-500}
\end{table}
\begin{table}[ht]
\centering
\small
\begin{tabular}{c|ccccc|ccccc|ccccc}
\toprule
& \multicolumn{5}{c|}{$s=1$}
& \multicolumn{5}{c|}{$s=2$}
& \multicolumn{5}{c}{$s=3$}\\
\cmidrule(lr){2-6}
\cmidrule(lr){7-11}
\cmidrule(lr){12-16}
$d$
& \multicolumn{5}{c|}{$\gamma$}
& \multicolumn{5}{c|}{$\gamma$}
& \multicolumn{5}{c}{$\gamma$}\\
& $.3$ & $.4$ & $.5$ & $.6$ & $.7$
& $.3$ & $.4$ & $.5$ & $.6$ & $.7$
& $.3$ & $.4$ & $.5$ & $.6$ & $.7$\\
\midrule
$.51$
& .078 & .070 & .065 & .066 & .067
& .084 & .069 & .069 & .068 & .066
& .087 & .079 & .072 & .071 & .061\\

$.55$
& .116 & .130 & .143 & .175 & .194
& .142 & .160 & .178 & .210 & .237
& .160 & .188 & .216 & .260 & .271\\

$.65$
& .264 & .358 & .455 & .549 & .625
& .382 & .503 & .619 & .739 & .816
& .461 & .606 & .731 & .835 & .900\\

$.75$
& .395 & .574 & .686 & .774 & .843
& .588 & .781 & .877 & .939 & .971
& .737 & .884 & .946 & .987 & .993\\

$1$
& .767 & .902 & .956 & .978 & .993
& .929 & .982 & .999 & 1 & 1
& .978 & .999 & 1 & 1 & 1\\

$1.25$
& .929 & .975 & .993 & .997 & .999
& .987 & .998 & 1 & 1 & 1
& .995 & 1 & 1 & 1 & 1\\
\bottomrule
\end{tabular}
\caption{Empirical rejection probabilities, at nominal level $\alpha=5\%$, for FARIMA(0,$d$,0) under $H_1$ with $n=5000$, 
$s\in\{1,2,3\}$ and
$m=\lfloor n^\gamma\rfloor$, where
$\gamma\in\{.3,.4,.5,.6,.7\}$. Number of replications is 1000.}
\label{tab:H1-5000}
\end{table}

\clearpage 

\subsection{Comparison with V/S test }
The V/S procedure requires estimating the memory parameter $d$ over a compact
subset of $(-1/2,\,1/2)$. More precisely, we chose the same compact $[-.4,.4]$ as suggested in \cite{MR2328526}. This restriction may induce boundary effects,
particularly when the true value of $d$ is close to $1/2$, and may partly
explain the inflated empirical size observed in this region.

Table~\ref{tab:VScomparison} shows that, for pure FARIMA$(0,d,0)$ processes,
the proposed statistic $Q_{n,m}(s)$ with the tuning parameters suggested in \eqref{tuning choice} maintains an empirical rejection
probability consistently close to the nominal level over the whole range of
values considered, whereas the empirical size of the V/S test increases
markedly as $d$ approaches $1/2$. The same behaviour is observed
for the larger sample size ($n=5000)$.

The same table also reports the corresponding results for FARIMA$(1,d,0)$
processes with positive and negative autoregressive coefficients. Although the
presence of short-range dependence slightly affects the empirical size of the
proposed test when $d$ is close to $1/2$, it remains substantially better
calibrated than the V/S procedure. Overall, these results indicate that the
proposed statistic is more robust to both long-range dependence and the
presence of a short-range autoregressive component.

Since the V/S test exhibits empirical rejection probabilities well above the
nominal level in several settings, comparing its empirical power with that of
the proposed procedure would not provide a fair assessment. Indeed, a test
with an inflated empirical size is expected to exhibit an artificially higher
power.
\begin{table}[ht]
\centering
\small
\begin{tabular}{r|ccccccccccccccc}
\toprule
\multicolumn{16}{c}{\textbf{FARIMA$(0,d,0)$, $n=500$}}\\
\midrule
$d$
& $-.49$ & $-.42$ & $-.35$ & $-.28$ & $-.21$
& $-.14$ & $-.07$ & $.00$ & $.07$ & $.14$
& $.21$ & $.28$ & $.35$ & $.42$ & $.49$\\
\midrule
$Q_{n,m}(s)$
& .02 & .03 & .04 & .02 & .03
& .02 & .04 & .03 & .02 & .04
& .03 & .03 & .03 & .03 & .05\\
V/S
& .05 & \textcolor{red}{.11} & \textcolor{red}{.10}
& \textcolor{red}{.09} & \textcolor{red}{.07}
& .03 & .03 & .02 & .01 & .02
& .02 & .02 & .05 & \textcolor{red}{.13}
& \textcolor{red}{.19}\\
\midrule

\multicolumn{16}{c}{\textbf{FARIMA$(0,d,0)$, $n=5000$}}\\
\midrule
$d$
& $-.49$ & $-.42$ & $-.35$ & $-.28$ & $-.21$
& $-.14$ & $-.07$ & $.00$ & $.07$ & $.14$
& $.21$ & $.28$ & $.35$ & $.42$ & $.49$\\
\midrule
$Q_{n,m}(s)$
& .02 & .03 & .03 & .04 & .04
& .03 & .04 & .03 & .05 & .05
& .03 & .03 & .04 & .04 & .05\\
V/S
& .02 & \textcolor{red}{.10} & \textcolor{red}{.10}
& .06 & .05
& .04 & .03 & .04 & .04 & .05
& .03 & .03 & .04 & \textcolor{red}{.12}
& \textcolor{red}{.23}\\
\midrule

\multicolumn{16}{c}{\textbf{FARIMA$(1,d,0)$ ($\phi=0.7$), $n=500$}}\\
\midrule
$d$
& $-.49$ & $-.42$ & $-.35$ & $-.28$ & $-.21$
& $-.14$ & $-.07$ & $.00$ & $.07$ & $.14$
& $.21$ & $.28$ & $.35$ & $.42$ & $.49$\\
\midrule
$Q_{n,m}(s)$
& .00 & .00 & .00 & .00 & .00
& .00 & .00 & .00 & .00 & .00
& .00 & .00 & .01 & .05 & \textcolor{red}{.13}\\
V/S
& .01 & .00 & .01 & .00 & .00
& .00 & .00 & .00 & .00 & .00
& .01 & .04 & \textcolor{red}{.08}
& \textcolor{red}{.18} & \textcolor{red}{.28}\\
\midrule

\multicolumn{16}{c}{\textbf{FARIMA$(1,d,0)$ ($\phi=-0.7$), $n=500$}}\\
\midrule
$d$
& $-.49$ & $-.42$ & $-.35$ & $-.28$ & $-.21$
& $-.14$ & $-.07$ & $.00$ & $.07$ & $.14$
& $.21$ & $.28$ & $.35$ & $.42$ & $.49$\\
\midrule
$Q_{n,m}(s)$
& .01 & .01 & .01 & .02 & .02
& .02 & .02 & .03 & .03 & .03
& .04 & .03 & .04 & .04 & \textcolor{red}{.07}\\
V/S
& .00 & .01 & .01 & .01 & .01
& .00 & .01 & .00 & .01 & .01
& .01 & .01 & .04 & \textcolor{red}{.10}
& \textcolor{red}{.19}\\
\bottomrule
\end{tabular}
\caption{Empirical rejection probabilities, at nominal level $\alpha=5\%$, of the proposed $Q_{n,m}(s)$ test and the V/S test with the recommended tuning parameters in \eqref{tuning choice}. Number of replications is 1000. Values greater than $6\%$ are shown in red.}
\label{tab:VScomparison}
\end{table}
\clearpage

\section{Proofs}

\subsection{Proof of Theorem \ref{Th1} }
\begin{proof}
The proof is based on the convergence of the random vector $Z_n(s,d)$
in \eqref{dft} and

the fact  that each denominator in $Q_{nm}(s,d)$ converges in
Probability to $D_{jj}$, more precisely 
For any $j=1,\ldots,s$, if $-1/2<d<3/2$ then, as $n\to\infty$,
\begin{equation}\label{mena limit}
\mathbb{E}\left[\left(\frac{1}{m}\sum_{h=1}^m\frac{I_{n,h}(\lambda'_j)}{\ell^{2d}}-D_{jj}\right)^2\right]\to 0.
\end{equation}

Indeed \eqref{mena limit} implies that $Q_{nm}(s,d)$ and
$\|D^{-1/2}Z_n\|^2(s,d)$ have the same asymptotic
distribution. Moreover, by  \eqref{dft} we have
$$
D^{-1/2}Z_n(s,d)\overset{d}{\longrightarrow}Z=\mathcal{N}\left(0,D^{-1/2}\Sigma(d)D^{-1/2}\right),
$$
and hence $\|D^{-1/2}Z_n\|^2$ converges in distribution to  $Z'Z$
whose  distribution is the sum of weighted independent $\chi^2(1)$
random variables given in Proposition \ref{Th1} since
$D^{-1/2}\Sigma(d)D^{-1/2}$ and $\Sigma(d)D^{-1}$ are similar matrices
and therefore have the same eigenvalues

In the rest of the proof we establish \eqref{mena limit} 
Observe that $\frac{I_{n,h}(\lambda'_j)}{(\lambda'_j)^{-1}}$
$h=1,\ldots,m$ are identically distributed, so that we have for fixed $j=1,\ldots,s$, as $n\to\infty$, 
$$
\mathbb{E}\left(\frac{1}{m}\sum_{h=1}^m\frac{I_{n,h}(\lambda'_j)}{\ell^{2d}}\right)=\mathbb{E}\left(\frac{I_{n,1}(\lambda'_j)}{\ell^{2d}}\right)\to D_{jj}
$$
by the proof of \eqref{dft} in \cite{OuldHayePhilippe2026} that relies on proving that $\textrm{Cov}(Z_n(s,d)\to\Sigma(d)$.
Hence it will be enough to show that, as $n\to\infty$,
\begin{equation}\label{lim-var}
    \textrm{Var}\left(\frac{1}{m}\sum_{h=1}^m\frac{I_{n,h}(\lambda'_j)}{\ell^{2d}}\right)\to0.
\end{equation}
We can write, for all $d\in -1/2, 3/2) $
 $$
I_{n,h}(\lambda_j')=\frac{1}{2\pi \ell}\left\vert\sum_{t=1}^{\ell}X_{t +(h-1)\ell}e^{\i t\lambda_j'}\right\vert^2$$ 

$$ 
I_{n,h}(\lambda_j')= \begin{cases}
    \frac{1}{2\pi \ell}\left\vert\sum_{t=1}^{\ell}S_t^{(h)}(Y)e^{\i t\lambda_j'}\right\vert^2,&\textrm{if }1/2\le d<3/2,\\\\
    \frac{1}{2\pi \ell}\left\vert\sum_{t=1}^{\ell}X_{t +(h-1)\ell}e^{\i t\lambda_j'}\right\vert^2,&\textrm{if }-1/2< d<1/2,
\end{cases}
$$

 where, for each $t$ such that  $1\le t\le \ell$,
 $$
 S_t^{(h)}(Y)=\sum_{u=1}^tY_{(h-1)\ell+u},\qquad h=1\ldots,m.
 $$
Let 
$$
U_{n,h}(d)=\begin{cases}
    \frac{1}{\ell^{1/2+d}}\displaystyle\sum_{k=1}^\ell\cos(k\lambda'_j)S_k^{(h)}(Y),&\textrm{if }1/2\le d<3/2,\\\\
    \frac{1}{\ell^{1/2+d}}\displaystyle\sum_{k=1}^\ell\cos(k\lambda'_j)X_{k+(h-1)\ell},&\textrm{if }-1/2< d<1/2,
\end{cases}
$$
and
$$
V_{n,h}(d)=\begin{cases}
    \frac{1}{\ell^{1/2+d}}\displaystyle\sum_{k=1}^\ell\sin(k\lambda'_j)S_k^{(h)}(Y),&\textrm{if }1/2\le d<3/2,\\\\
    \frac{1}{\ell^{1/2+d}}\displaystyle\sum_{k=1}^\ell\sin(k\lambda'_j) X_{k+(h-1)\ell},&\textrm{if }-1/2< d<1/2,
\end{cases}
$$
and let

$$
d_{\ell,h,u}^{(c)}=\begin{cases}
\frac{1}{\ell^{1/2+d}}\displaystyle\sum_{k=1}^\ell\left(\sum_{i=1}^ka_{(h-1)\ell+i-u}\right)\cos\left(\frac{2\pi jk}{\ell}\right),&\textrm{if }1/2\le d<3/2,\\\\ 
\frac{1}{\ell^{1/2+d}}\displaystyle\sum_{k=1}^\ell a_{(h-1)\ell+k-u}\cos\left(\frac{2\pi jk}{\ell}\right),&\textrm{if }-1/2<d<1/2.
\end{cases}
$$
and we note $d_{\ell,u}=d_{\ell,u}^{(c)} = d_{\ell,1,u}^{(c)} $ and $d_{\ell,h,u}=d_{\ell,h,u}^{(c)}$.

Using the stationarity of $X_k$ (when $-1/2<d<1/2$) and $Y_k=X_k-X_{k-1}$ (when $1/2\le d<3/2)$,  the left hand side of (\ref{lim-var}) can then be written as
\begin{eqnarray*}\lefteqn{
\frac{1}{m^2}\left[\sum_{i=1}^m\sum_{j=1}^m\textrm{Cov}(U_{n,i}^2(d),U_{n,j}^2(d))+\sum_{i=1}^m\sum_{j=1}^m\textrm{Cov}(V_{n,i}^2(d),V_{n,j}^2(d))+2\sum_{i=1}^m\sum_{j=1}^m\textrm{Cov}(U_{n,i}^2(d),V_{n,j}^2(d))\right]}\\
&&=\frac{1}{m}\left[\textrm{Var}(U^2_{n,1}(d))+\textrm{Var}(V_{n,1}^2)+2\textrm{Cov}(U^2_{n,1}(d),V_{n,1}^2(d))\right]+\Bigg[\frac{2}{m}\sum_{h=2}^{m-1}\left(1-\frac{h-1}{m}\right)\\
&&\left[\textrm{Cov}(U^2_{n,1}(d),U^2_{n,h}(d))+\textrm{Cov}(V_{n,1}^2(d),V^2_{n,h}(d))+\textrm{Cov}(U^2_{n,1}(d),V^2_{n,h}(d))+\textrm{Cov}(V_{n,1}^2(d),U^2_{n,h}(d))\right]\Bigg].
\end{eqnarray*}
We will show that for $-1/2<d<1/2$,
\begin{equation}\label{cov-lim0}
    \textrm{Cov}(U^2_{n,1}(d),U_{n,h}^2(d))\le C\left(\ell^{-1}+h^{2d-1}\right).
\end{equation}
The other covarainces treat similarly and this will complete the proof of (\ref{lim-var}). The case $1/2\le d<3/2$ treats similarly by writing $d=d-1+1$.\\
{\bf Proof of (\ref{cov-lim0}):}
We have for $h\ge2$, 
$$
\textrm{Cov}(U^2_{n,1}(d),U_{n,h}^2(d))=\mathbb{E}(U^2_{n,1}(d)U^2_{n,h}(d))-\left(\mathbb{E}(U^2_{n,1}(d))\right)^2,
$$
\begin{eqnarray}\label{cov-cov}\lefteqn{ 
\mathbb{E}(U^2_{n,1}(d)U_{n,h}^2(d))=\mathbb{E}\left[\left(\sum_{u=-\infty}^\ell d_{\ell,u}\epsilon_u\right)^2\left(\sum_{u=-\infty}^{h\ell} d_{\ell,h,u}\epsilon_u\right)^2\right]}\nonumber\\
&&=\left(\mathbb{E}(U^2_{n,1}(d))\right)^2+
\textrm{Var}(\epsilon_1^2)\sum_{u=-\infty}^\ell d^2_{\ell,u}d^2_{\ell,h,u}+2(\mathbb{E}(\epsilon_1^2))^2\sum_{u\neq v=-\infty}^\ell d_{\ell,u}d_{\ell,v}d_{\ell,h,u}d_{\ell,h,v}. 
\end{eqnarray}
We can easily see that for $u\le \ell$, and $d\neq0$,
$$
d^2_{\ell,h,u}\le\frac{C}{\ell^{1+2d}}\left((h\ell-u)^d-((h-1)\ell-u)^d\right)^2
$$
and for $u<0,$
\begin{equation}\label{eqq1}
d_{u,\ell}^2\le\frac{C}{\ell^{1+2d}}\left(\ell-u)^d-(-u)^d\right)^2.
\end{equation}
Hence, uniformly in $h\ge2$, using the fact that when $-\ell\le u\le-1$,
$$
d^2_{\ell,h,u}\le C\frac{1}{\ell},
$$
\begin{eqnarray*}
\sum_{u=-\ell}^{-1} d^2_{\ell,u}d^2_{\ell,h,u}&\le&
\frac{C}{\ell}\sum_{u=1}^\ell\left[\left(1+\frac{u}{\ell}\right)^d-\left(\frac{u}{\ell}\right)^d\right]^2\frac{1}{\ell}\le\frac{C}{\ell}\int_0^1\left(1+x)^d-x^d\right)^2dx\to0,
\end{eqnarray*}
$$
\sum_{u=-\infty}^{-\ell} d^2_{\ell,u}d^2_{\ell,h,u}\le
\frac{C}{\ell^{2d}}\sum_{u=\ell}^\infty(\ell+u)^{2d-2}\le C\ell^{-1},
$$
and for $0\le u\le\ell$, 
\begin{equation}\label{eqq2}
d^2_{\ell,u}\le\frac{1}{\ell^{1+2d}}\left(\sum_{i=u}^\ell \vert a_{i-u}\vert\right)^2\le \frac{C}{\ell^{1+2d}}(1+\ell-u)^{2d}, 
\end{equation}
so that, for $h\ge3$,
$$
\sum_{u=0}^\ell d^2_{\ell,u}d^2_{\ell,h,u}\le\frac{C}{\ell^{4d}}\sum_{k=1}^\ell k^{2d}(k+\ell)^{2d-2}\le \frac{C}{\ell}.
$$
In summary, when $d\neq0$, we obtained that, uniformly in $h\ge3$, as $\ell\to\infty$
\begin{equation}\label{simple no zero}
    \sum_{s=-\infty}^\ell d^2_{\ell,u}d^2_{\ell,h,u}\to0.
\end{equation}
For the cross term in \eqref{cov-cov}, using \eqref{eqq1}
and \eqref{eqq2}, we can easily see that
\begin{equation}\label{finite 2nd moment}
\sum_{u=-\infty}^\ell d^2_{\ell,u}\le C
\end{equation}
and we obtain, for $h\ge3$, uniformly in $\ell,$
\begin{eqnarray}\label{no zero}
    \sum_{u\neq v=-\infty}^\ell d_{\ell,u}d_{\ell,v}d_{\ell,h,u}d_{\ell,h,v}&\le&
\left(\sum_{u=-\infty}^\ell d^2_{\ell,u}\right)\left(\sum_{u=-\infty}^\ell d^2_{\ell,h,u}\right)\nonumber\\
&\le& C\left(\sum_{u=-\infty}^\ell d^2_{\ell,h,u}\right)\le C(h-1)^{2d-1}.
\end{eqnarray}
The proof of \eqref{simple no zero}
 and \eqref{no zero} is more straightforward  when $d=0$ using the summability of $a_j$, which completes the proof of \eqref{cov-lim0}.
\end{proof}     
\subsection{Proof of Proposition \ref{continuity1}}
\begin{proof}
We first show that,  when  $0<d<1/2$, we can rewrite the formulae defining $\Sigma(d)$ in \eqref{Sigma} under the form
\begin{equation}\label{other-form1}
\Sigma^{(c)}_{ij}=C(d)\int_{[0,1]^2}
\cos(2\pi i x)\cos(2\pi j y)\vert x-y\vert^{2d-1}dxdy,
\end{equation}
and
\begin{equation}\label{other-form2}
\Sigma^{(s)}_{ij}=C(d)\int_{[0,1]^2}
\sin(2\pi i x)\sin(2\pi j y)\vert x-y\vert^{2d-1}dxdy.
\end{equation}
In fact, Consider first $\Sigma^{(c)}_{ij}(d):$
Using Leibniz formula: if $f$ is differentiable under the integral
$$
\frac{\partial}{\partial x}\int_a^xf(y,x)dy=f(x,x)+\int_a^x\frac{\partial}{\partial x}f(y,x)
$$
and
$$
\frac{\partial}{\partial y
}\int_y^af(y,x)dx=-f(x,x)+\int_y^a\frac{\partial}{\partial y}f(y,x)dx,
$$
we get, since here we will have $f(x,x)=0$,
\begin{eqnarray*}\lefteqn{
I= \frac{(2\pi i)(2\pi j)}{2} \int_{[0,1]^2}
\sin(2\pi i x)\sin(2\pi j y)\vert x-y\vert^{2d+1}dxdy}\\
&&=\frac{(2\pi i)(2\pi j)}{2}\int_0^1\int_0^x\left[\sin(2\pi i x)\sin(2\pi j y)+\sin(2\pi j x)\sin(2\pi i y)\right](x-y)^{2d+1}dydx.
\end{eqnarray*}
\begin{eqnarray*}\lefteqn{
  (2\pi i)(2\pi j) \int_0^1\int_0^x\sin(2\pi i x)\sin(2\pi j y)(x-y)^{2d+1}dydx}\\
&&=\int_0^1 \left(\int_0^x (2\pi j)\sin(2\pi j y)(x-y)^{2d+1}\,dy\right)(2\pi i)\sin(2\pi ix)\,dx:=\int_0^1h_j(x)(2\pi i)\sin(2\pi ix)\,dx\\
&&=
\left[h_j(x)(-\cos(2\pi ix)\right]_0^1+\int_0^1h_j'(x)\cos(2\pi ix)dx\\
&&=-h_j(1)+(2d+1)\int_0^1\left(\int_0^x2\pi j(\sin(2\pi jy)(x-y)^{2d}dy\right)\cos(2\pi ix)dx\\
&&=\int_0^1(-2\pi j)\sin(2\pi jy)(1-y)^{2d+1}dy+(2d+1)\int_0^1\left(\int_y^1\cos(2\pi ix)(x-y)^{2d}dx\right)2\pi j\sin(2\pi jy)dy\\
&&:=
\int_0^1(2\pi j)\sin(2\pi jy)y^{2d+1}dy+(2d+1)\int_0^1\ell(y)2\pi j\sin(2\pi jy)dy\\
&&=-1+(2d+1)\left[\int_0^1\cos(2\pi jy)y^{2d}dy+\left[\ell_i(y)(-\cos(2\pi jy))\right]_0^1+\int_0^1\ell_i'(y)\cos(2\pi jy)dy\right]\\
&&
=-1+(2d+1)\left[\int_0^1\cos(2\pi jy)y^{2d}dy+\ell_i(0)\right]\\
&&-2d(2d+1)\int_0^1\left(\int_y^1\cos(2\pi ix)(x-y)^{2d-1}dx\right)\cos(2\pi jy)dy.
\end{eqnarray*}
Therefore, by symmetry (in $i$ and $j$), 
\begin{eqnarray*}\lefteqn{
I=-1+
(2d+1)\left[\int_0^1(\cos(2\pi jy)+\cos(2\pi iy))y^{2d}dy\right]}\\
&&-d(2d+1)\int_{[0,1]^2}\cos(2\pi ix)\cos(2\pi jy)\vert x-y\vert^{2d-1}dxdy\\
=&&-a_{ij}(d)-\frac{C(d)}{\delta(d)}\int_{[0,1]^2}\cos(2\pi ix)\cos(2\pi jy)\vert x-y\vert^{2d-1}dxdy
\end{eqnarray*}
Therefore
$$
\Sigma^{(c)}_{ij}(d)=-\delta(d)(I+a_{ij}(d))=C(d)\int_{[0,1]^2}\cos(2\pi ix)\cos(2\pi jy)\vert x-y\vert^{2d-1}dxdy,
$$which implies (\ref{other-form1}). The formula (\ref{other-form2}) is much simpler as the bracket term, in the integration by parts, is zero, and we are omitting its proof. 
We now prove the continuity. With $v=2d-1$ and $h=\vert x-y\vert$ fixed and $f(v)=h^v$, we can write (Taylor to the first order  with integral remainder
\begin{eqnarray*}
f(v)&=&\exp(v\log h)=1+v\log h+(\log h)^2\int_0^v(v-u)f(u)du\\
&=&1+v\log h+v^2(\log h)^2\int_0^1(1-t)f(ts)dt,
\end{eqnarray*}
and we have for $v<0$,  $0<h\le1$, $ts\log h\le v\log h$, so that 
$$
\int_0^1(1-t)f(tv)dt\le\int_0^1f(v)dt= h^v.
$$It follows that for  fixed $i,j$,
\begin{align*}
&\left\vert\int_{[0,1]^2}
\cos(2\pi i x)\cos(2\pi j y)\vert x-y\vert^{2d-1}dxdy-v\int_{[0,1]^2}
\cos(2\pi i x)\cos(2\pi j y)\log(\vert x-y\vert dxdy\right\vert\\
&\leq v^2\int_{[0,1]^2}(\log \vert x-y\vert)^2\vert x-y\vert^vdxdy\\
&\le v^2\left(\int_{[0,1]^2}(\log \vert x-y\vert)^4dxdy\int_{[0,1]^2}\vert x-y\vert^{2v}dxdy\right)^{1/2}=v^2\left(\frac{48}{(2v+1)(v+1)}\right)^{1/2}\\
&=O(v^2),\quad\textrm{as }v\to0.
\end{align*}
That is, we get
$$
\frac{\delta(-1/2)}{2}\frac{1}{C(d)}\Sigma^{(c)}_{ij}(d)=(2d-1)\Sigma^{(c)}_{ij}(1/2)+O(2d-1)^2
$$
and similarly for $\Sigma^{(s)}_{ij}$,
and hence the entries of the limiting normalized covariance matrix satisfy
\begin{eqnarray*}
\left(\Sigma(d)D^{-1}(d)\right)_{i,j}&=&
\frac{\Sigma_{ij}(d)}{\Sigma^{(c)}_{ii}(d)+\Sigma^{(s)}_{ii}(d)}
=
\frac{\Sigma_{ij}(1/2)+O(2d-1)}{\Sigma^{(c)}_{ii}(1/2)+\Sigma^{(s)}_{ii}(1/2)+O(2d-1)}\\\\
&\to&\left(\Sigma(1/2)D^{-1}(1/2)\right)_{i,j}\quad\textrm{as }d\to1/2,
\end{eqnarray*}
as $\Sigma^{(c)}_{ii}(1/2)$ and $\Sigma^{(s)}_{ii}(1/2)$ are positive.\end{proof}
\subsection{Proof  of   Proposition \ref{prop2}} 
\label{ap:proof2}
As $\delta \leq 1/2$ appears only in the normalization $m^{-2\delta}$ of the statistic $Q_{n.m}(s,\delta)$, it is enough to  prove the convergence \eqref{AsympH1} for $\delta = 1/2$, i.e. $Q_{n,m}(s,1/2)\overset{P}{\longrightarrow}\infty$.

For a stochastic or deterministic trend, it is enough to establish that
$$
\left(m^{-1}\right)\sum_{j=1}^s \frac{I_n(\lambda_j)}{\frac{1}{m}\sum_{i=1}^mI_{n,i}(\lambda'_j)}\overset{P}{\to}\infty.
$$
This will follow if, for fixed $j=1,\ldots,s$, we can prove 
\begin{equation}\label{loin}
\left(m^{-1}\right) \frac{I_n(\lambda_j)}{\frac{1}{m}\sum_{i=1}^mI_{n,i}(\lambda'_j)}\overset{P}{\to}\infty.
\end{equation}
{\bf Proof of (\ref{loin}) in the deterministic trend case $\mathbf{X_t=g_n(t)+Y_t}$.}\\
For fixed $j=1,\ldots,s$, and omitting the coefficient $2\pi$ in the denominator, we have from (\ref{perio}), with $I_{n,Y}$ denoting the periodogram built from $Y_1,\ldots,Y_n$,
\begin{eqnarray}\label{decomp}
I_n(\lambda_j)&=&I_{n,Y}(\lambda_j)+n^{2\beta+1}\vert D_{n,g}(j)\vert^2
+2n^{\beta+1/2}\,\,\textrm{Re}\left(D_{n,g}(j)\overline{D_{n,Y}(j)}\right)
\end{eqnarray}
where
$$
D_{n,g}(j)=\frac{1}{\sqrt{2\pi}}\sum_{t=1}^ng\left(\frac{t}{n}\right)e^{\i 2\pi t j/n}\frac{1}{n},
$$
and
$$
D_{n,Y}(j)=\frac{1}{\sqrt{2\pi n}}\sum_{t=1}^nY_te^{\i 2\pi t j/n}.
$$
Since $Y_t$ satisfies Theorem \ref{Th1}'s assumptions, we have as $n\to\infty$, 
$$
\mathbb{E}(I_{n,Y}(\lambda_j))\sim L_j(d) f(\lambda_j)\sim L_j(d) n^{2d}.
$$
Under assumption (\ref{nonzero}), we obtain, via a Riemann sum approximation of the integral,
$$
\vert D_{n,g}(j)\vert^2\sim c>0, \textrm{ as }n\to\infty,
$$
where $c$ is a positive constant that may change from one expression to another.
Finally, for the cross-term, we have
\begin{eqnarray*}\label{cross}
\mathbb{E}\vert D_{n,g}(j)\overline{D_{n,Y}(j)}\vert&=&\vert D_{n,g}(j)\vert\,\mathbb{E}\vert D_{n,Y}(j)\vert\\
&\le&\vert D_{n,g}(j)\vert\,\mathbb{E}^{1/2}\left(\vert D_{n,Y}(j)\vert^2\right)
=\vert D_{n,g}(j)\vert\,\mathbb{E}^{1/2}(I_{n,Y}(\lambda_j))\le cn^d.
\end{eqnarray*}
Therefore
\begin{equation}\label{num1}
\frac{I_n(\lambda_j)}{ n^{2\beta+1}}\overset{P}{\to}c.
\end{equation}
We now analyse the behaviour of the denominator term
$$
\frac{1}{m}\sum_{i=1}^mI_{n,i}(\lambda'_j)=\frac{1}{m}\sum_{i=1}^m\frac{1}{2\pi\ell}\left\vert\sum_{t=(i-1)\ell+1}^{i\ell}(Y_t+g_n(t))e^{\i t\lambda'_j}\right\vert^2.
$$
Following the same decomposition as in (\ref{decomp}), we identify three components to examine: a  stochastic term, a deterministic trend term, and a cross term.\\
Denote $I_{n,i,Y}$  the periodogram  built from the $i$th epoch of  the stationary process $Y_t$.
Using (\ref{lim-var}), we obtain
$$
\frac{1}{m}\sum_{i=1}^mI_{n,i,Y}(\lambda'_j)=O_P(\ell^{2d}).
$$
For the trend part,  we have
\begin{eqnarray}\label{break}\lefteqn{
\frac{1}{m}\sum_{u=1}^m\frac{1}{\ell}\left\vert\sum_{t=1}^\ell g_n(t+(u-1)\ell)e^{\i t\lambda'_j}\right\vert^2}\nonumber\\
&&=\ell^{-1}n^{2\beta}
\sum_{t=1}^\ell\sum_{s=1}^\ell\left(\sum_{u=1}^mg\left(\frac{t+(u-1)\ell}{n}\right)
g\left(\frac{s+(u-1)\ell}{n}\right)\frac{1}{m}\right)
e^{\i (t-s)\lambda'_j}\nonumber\\
&&\sim\ell^{-1}n^{2\beta}
\sum_{t=1}^\ell\sum_{s=1}^\ell\left[\int_0^1g\left(\frac{t}{n}+x\right)g\left(\frac{s}{n}+x\right)dx\right]e^{\i (t-s)\lambda'_j}\nonumber\\
&&=\ell^{-1}n^{2\beta}\int_0^1\left(\left\vert\sum_{t=1}^\ell g\left(\frac{t}{n}+x\right)e^{\i t\lambda'_j}\right\vert^2\right)dx.
\end{eqnarray}
(In the integrals above we extended $g$ beyond  (0,1) with value 0).
If $g$ is continuous over [0,1] then it is uniformly continuous and hence 
$$
g\left(\frac{t}{n}+x\right)=g(x)+\epsilon_n(t,x)
$$
where $\epsilon_n(t,x)\to0$ uniformly in $t=1,\ldots,\ell$ and $x$, and hence, since 
\begin{equation}\label{sumzero1}
\sum_{t=1}^\ell e^{\i t\lambda_j'}=0,
\end{equation}
the integrand in the last expression in (\ref{break}) is $o(\ell^2)$. Therefore, we have
$$
\frac{1}{m}\sum_{u=1}^m\frac{1}{\ell}\left\vert\sum_{t=1}^\ell g_n(t+(u-1)\ell)e^{\i t\lambda'_j}\right\vert^2=o\left(n^{2\beta}\ell\right).
$$
If $g$ is piece-wise continuous on [0,1] then there exists $M$ sub-intervals  $(A_k,B_k)$, $j=1,\ldots,M$,  such that $g$ is continuous on such intervals with finite limits at $A_k$ and $B_k$. The previous result (when $g$ is continuous) remains true for each interval $[A_k,B_k]$.    Hence we find that  (\ref{break}) is still bounded by $Mn^{2\beta}o(\ell)=o\left(n^{2\beta}\ell\right)$.

For the cross-term, and similarly to (\ref{cross}),  we easily obtain  that
\begin{eqnarray*}\lefteqn{
\mathbb{E}\left(\frac{1}{m}\sum_{u=1}^m
\frac{1}{\ell}\sum_{t=1}^\ell\left\vert g_n\left(t+(u-1)\ell\right)e^{\i t\lambda'_j}\right\vert
\left\vert\sum_{t=1}^\ell Y_{t+(u-1)\ell} e^{\i t\lambda'_j}\right\vert\right)}\\
&&=\frac{1}{n}
\ell^{1/2}\sum_{u=1}^m\sum_{t=1}^\ell\left\vert g_n\left(t+(u-1)\ell\right)e^{\i t\lambda'_j}\right\vert
\mathbb{E}\left[\frac{1}{\sqrt{\ell}}\left\vert\sum_{t=1}^\ell Y_{t+(u-1)\ell} e^{\i t\lambda'_j}\right\vert\right]\\
&&=
\ell^{1/2}\left(\sum_{u=1}^m\sum_{t=1}^\ell\left\vert g_n\left(t+(u-1)\ell\right)e^{\i t\lambda'_j}\right\vert\frac{1}{n}\right)
\mathbb{E}\left[\frac{1}{\sqrt{\ell}}\left\vert\sum_{t=1}^\ell Y_t e^{\i t\lambda'_j}\right\vert\right]\\
&&\le\ell^{1/2}n^\beta\sum_{k=1}^n\left\vert g\left(\frac{k}{n}\right)\right\vert\frac{1}{n}\,\,\mathbb{E}^{1/2}\left(I_{n,1,Y}(\lambda'_j)\right)\\
&&\sim n^\beta\ell^{1/2}\,\,\int_0^1\vert g(x)\vert dx\,\,\left[L_j(d)f(\lambda'_j)\right]^{1/2}\\
&&=cn^\beta\ell^{1/2+d}.
\end{eqnarray*}
Finally, the denominator is of order $\ell \,n^{2\beta}\,o_P(1)$. Therefore, using (\ref{num1}), for each fixed frequency $\lambda_j$, 
$$
\left(\frac{n}{\ell}\right)^{-1}\frac{I_n(\lambda_j)}{\frac{1}{m}\sum_{i=1}^mI_{n,i}(\lambda'_j)}=
\frac{\ell}{n}\frac{n^{2\beta+1}}{n^{2\beta}\,\,\ell\,\, o_P(1)}=\frac{1}{o_P(1)}
\overset{P}{\to}\infty,
$$
which implies that $Q_{n,m}(s,1/2)$ goes to infinity as $n\to\infty$.
\subsection{Proof of Proposition \ref{prop22}}
\begin{proof} Let  $s$ be fixed and let $d_n$ be an arbitrary sequence in $(-1/2,1/2]$ converging to $d\in(-1/2,1/2]$. Using Proposition \ref{continuity1} and since $Q(s,d)$ is a quadratic form of a Gaussian vector then we obtain that $Q(s,d_n)$ converges to $Q(s,d)$ in distribution. Since $Q(s,d)$ has a continuous  distribution function $F_{s,d}$, by Polya's theorem (see e.x. \cite{lehmann1999elements}, Theorem 2.6.1), $$\|F_{s,d_n}-F_{s,d}\|\to0,$$ 
where $\|.\|$ is the supremum norm.
Moreover, the inversibility of $F_{s,d}$ implies that (with $F=F_{s,d}$ and $F_n=F_{s,d_n}$)
$$
F\left(F^{-1}_n(\alpha)\right)=\left\vert F\left(F^{-1}_n(\alpha)\right)-F_n\left(F^{-1}_n(\alpha)\right)+
F_n\left(F^{-1}_n(\alpha)\right)\right\vert\le\|F_n-F\|+\alpha
\to\alpha$$
which means that $F_n^{-1}(\alpha)\to F^{-1}(\alpha)$ since $F^{-1}$ is continuous. This completes the proof.
\end{proof}
\subsection{Proof of Proposition \ref{empirical size 1}}
\begin{proof}
{\bf Asymptotic probability of $\mathbf{R_n}$ under $H_0$ }\\
According to Theorem \ref{Th1} and the properties of the local Whittle estimator defined in (\ref{lwe}) we have 
$$
Q_{n,m}(s,\widehat d)\overset{\mathcal{D}}{\to} Q(s,d).
$$
As $Q(s,d)$ has a continuous cumulative distribution function (cdf) $F_{d,s}$ (see proof of Proposition \ref{prop22}), the cdf of $ Q_{n,m}(s,\widehat d)$ converges uniformly to the cdf of  $Q(s,d)$ 
$$ P(Q_{m,n}(s,\widehat d)>q_\alpha(s,\widehat d))  \sim  1 - F_{d,s} (q_\alpha(s,\widehat d)) , \quad \text{as } n\to\infty$$
Moreover, according to Proposition \ref{prop22}, the function  $d\mapsto q_\alpha(s,d)$ is continuous, and hence  
\begin{equation}\label{limit alpha1}
P_d(Q_{m,n}(s,\widehat d)>q_\alpha(s,\widehat d))  \to \alpha , \quad \text{as } n\to\infty.
\end{equation}

{\bf Asymptotic probability of $\mathbf{R_n}$ under $H_1$  } We have 
 $$
  Q_{n,m}(s,\widehat d)=e^{-2(\widehat d-1/2  )\ln m}Q_{n,m}(s,1/2) \geq Q_{n,m}(s,1/2) \overset{P}{\longrightarrow} \infty,
  $$
by Proposition \ref{Th1}.\end{proof}

\bibliographystyle{apalike}
\bibliography{jtsa_new}

@unpublished{OuldHayePhilippe2026,
  author    = {Mohamedou {Ould Haye} and Anne Philippe},
  title     = {From nonstationarity to stationarity via {$1/f$} noise:
               discrete Fourier transforms and sample mean asymptotics for testing},
  year      = {2026},
  note      = {Preprint, 	arXiv:2605.28339}

}

@book{lehmann1999elements,
  title={Elements of large-sample theory},
  author={Lehmann, Erich Leo},
  year={1999},
  publisher={Springer}
}

@article{pptest,
 doi       = {10.2307/2336182},
 title     = {Testing for a Unit Root in Time Series Regression},
 author    = {Peter C. B. Phillips and Pierre Perron},
 publisher = {Oxford University Press},
 journal   = {Biometrika},
 issn      = {0006-3444,1464-3510},
 year      = {1988},
 volume    = {75},
 issue     = {2},
 pages     = {335--346},
 url       = {http://doi.org/10.2307/2336182}
}

@Manual{R,
    title = {R: A Language and Environment for Statistical Computing},
    author = {{R Core Team}},
    organization = {R Foundation for Statistical Computing},
    address = {Vienna, Austria},
    year = {2018},
    url = {https://www.R-project.org/},
  }

@article {MR3505787,
    AUTHOR = {Bailey, Natalia and Giraitis, Liudas},
     TITLE = {Spectral approach to parameter-free unit root testing},
   JOURNAL = {Comput. Statist. Data Anal.},
  FJOURNAL = {Computational Statistics \& Data Analysis},
    VOLUME = {100},
      YEAR = {2016},
     PAGES = {4--16},
      ISSN = {0167-9473},
   MRCLASS = {62M15 (60G12 62G10 62M10)},
  MRNUMBER = {3505787},
MRREVIEWER = {Zuzana Pr\'{a}\v{s}kov\'{a}},
       DOI = {10.1016/j.csda.2015.05.002},
       URL = {https://doi.org/10.1016/j.csda.2015.05.002},
}

@article {Dalla,
    AUTHOR = {Dalla, Violetta and Giraitis, Liudas and Hidalgo, Javier},
     TITLE = {Consistent estimation of the memory parameter for nonlinear
              time series},
   JOURNAL = {J. Time Ser. Anal.},
  FJOURNAL = {Journal of Time Series Analysis},
    VOLUME = {27},
      YEAR = {2006},
    NUMBER = {2},
     PAGES = {211--251},
      ISSN = {0143-9782},
   MRCLASS = {62M07 (62M15)},
  MRNUMBER = {2235845},
MRREVIEWER = {Guy Jumarie},
       DOI = {10.1111/j.1467-9892.2005.00464.x},
       URL = {https://doi.org/10.1111/j.1467-9892.2005.00464.x},
}

@book {MR2977317,
    AUTHOR = {Giraitis, Liudas and Koul, Hira L. and Surgailis, Donatas},
     TITLE = {Large sample inference for long memory processes},
 PUBLISHER = {Imperial College Press, London},
      YEAR = {2012},
     PAGES = {xvi+577},
      ISBN = {978-1-84816-278-5; 1-84816-278-2},
   MRCLASS = {62-02 (60G10 62Mxx)},
  MRNUMBER = {2977317},
MRREVIEWER = {Gilles Teyssi\`ere},
       DOI = {10.1142/p591},
       URL = {https://doi.org/10.1142/p591},
}

@article {MR2328526,
    AUTHOR = {Giraitis, Liudas and Leipus, Remigijus and Philippe, Anne},
     TITLE = {A test for stationarity versus trends and unit roots for a
              wide class of dependent errors},
   JOURNAL = {Econometric Theory},
  FJOURNAL = {Econometric Theory},
    VOLUME = {22},
      YEAR = {2006},
    NUMBER = {6},
     PAGES = {989--1029},
      ISSN = {0266-4666},
   MRCLASS = {Expansion},
  MRNUMBER = {2328526},
       DOI = {10.1017/S026646660606049X},
       URL = {https://doi.org/10.1017/S026646660606049X},
}

@Article{Gromykov,
  author={Gromykov, Gennadi and Ould Haye, Mohamedou and Philippe, Anne },
  title={{A frequency-domain test for long range dependence}},
  journal={Statistical Inference for Stochastic Processes},
  year=2018,
  volume={21},
  number={3},
  pages={513-526},
  month={October},
  doi={10.1007/s11203-017-9164-6},
  url={https://ideas.repec.org/a/spr/sistpr/v21y2018i3d10.1007_s11203-017-9164-6.html}
}
\end{document}